\documentclass[12 pt]{amsart}
\usepackage{graphicx}
\usepackage{booktabs}

\usepackage[dvipsnames]{xcolor}
\usepackage{url}

\usepackage{tikz, tikz-cd, varwidth, amssymb}
\usepackage[hidelinks]{hyperref}
\usetikzlibrary{calc, decorations}
\usepackage[text={6 in, 7.5 in}, centering]{geometry}
\usepackage{enumerate, enumitem, stmaryrd}
\usepackage[all,cmtip]{xy}

\usepackage{mathdots}

\newcommand{\CC}{\mathbb{C}}

\newcommand{\FF}{\mathbb{F}}

\newcommand{\HH}{\mathbb{H}}

\newcommand{\PP}{\mathbb{P}}
\newcommand{\QQ}{\mathbb{Q}}

\newcommand{\ZZ}{\mathbb{Z}}

\newcommand{\cA}{\mathcal{A}}

\newcommand{\cM}{\mathcal{M}}

\newcommand{\s}{\mathsf{S}}

\newcommand{\ch}{\operatorname{char}}
\newcommand{\cyc}{\operatorname{cyc}}

\newcommand{\Gal}{\operatorname{Gal}}

\newcommand{\Mbar}{\overline{\mathcal{M}}}

\newcommand{\Sp}{\operatorname{Sp}}
\newcommand{\Spec}{\operatorname{Spec}}
\newcommand{\Sym}{\operatorname{Sym}}

\newcommand{\SL}{\mathrm{SL}}

\newtheorem{Theorem}{Theorem}

\newtheorem{theorem}{Theorem}[section]
\newtheorem{proposition}[theorem]{Proposition}
\newtheorem{problem}[theorem]{Problem}

\newtheorem*{HC}{Hodge Conjecture}
\newtheorem*{GHC}{Generalized Hodge Conjecture}
\newtheorem*{TC}{Tate Conjecture}
\newtheorem*{GTC}{Generalized Tate Conjecture}
\newtheorem*{SSC}{Semisimplicity Conjecture}
\newtheorem*{Taut20}{Taut 20 Conjecture}
\newtheorem*{Taut10}{Taut 10 Conjecture}
\newtheorem*{Chow5}{Chow 5 Conjecture}

\theoremstyle{definition}
\newtheorem{remark}[theorem]{Remark}

\title{On the Hodge and Tate conjectures for moduli spaces of curves}
\author{Sam Payne}

\begin{document}

\begin{abstract}
We survey recent progress on the cohomology of moduli spaces of stable curves through the lens of the Hodge and Tate conjectures, especially their generalized coniveau forms, which relate Hodge structures and $\ell$-adic Galois representations on cohomology to algebraic cycles. We explain how the inductive structure of the boundary stratification verifies these conjectures in a surprisingly wide range of cases, describe the guiding inspiration from arithmetic, and discuss open problems and directions for future research.
\end{abstract}

\maketitle

\section{Introduction}

The Hodge conjecture and Tate conjecture are two central threads in a web of interrelated conjectures relating cohomology theories for algebraic varieties to the properties of algebraic cycles and categories of motives. These conjectures are wide open in general, and known in only a few special cases. Progress has typically involved the study of varieties that move in nontrivial families, such as abelian varieties and K3 surfaces, using degeneration arguments and the properties of special varieties with extra endomorphisms, such as CM abelian varieties.

Recent progress in understanding the cohomology of moduli spaces of stable curves opens a new avenue for the investigation of these conjectures. When $g$ and $n$ are large, $\Mbar_{g,n}$ is typically of general type, admits nonzero holomorphic forms, and has huge Chow groups. These moduli spaces do not deform in families, but instead constitute an infinite discrete collection of spaces whose cohomology groups contain rich and complicated Hodge structures and Galois representations. The geometric role usually played by variation in families is here played by the tautological gluing and forgetful morphisms, together with the boundary stratification and weight arguments.

The recent cohomology computations are discussed in more detail in Larson's survey article \cite{Larson26}. The purpose of this article is to explain their consequences for the Hodge and Tate conjectures, especially in their generalized coniveau forms.

\subsection{Main statements} Throughout, $\Mbar_{g,n}$ denotes the smooth and proper Deligne--Mumford moduli stack of stable $n$-pointed genus $g$ curves, and $\cM_{g,n} \subset \Mbar_{g,n}$ is the open substack of smooth curves. Except where explicitly stated otherwise, cohomology is rational singular cohomology over $\CC$, or $\ell$-adic cohomology after base change to a separable closure.

The inductive structure of the stratification by topological types of stable curves, combined with topological properties of the locally closed strata and some basic mixed Hodge theory, allows for the verification of the generalized Hodge and Tate conjectures in a surprisingly wide range of cases.

Throughout, we assume that $g$ and $n$ are nonnegative integers such that $2g-2+n > 0$, so the moduli space of stable curves $\Mbar_{g,n}$ is nonempty.

\begin{Theorem} \label{thm:a}
The Hodge conjecture holds on $\Mbar_{g,n}$ for:
\begin{enumerate}
    \item algebraic cycles of codimension 1 and 2, for all $g$ and $n$,
    \item algebraic cycles of codimension 3, for $g \not \in \{7, 8, 9\}$, and for all $n$,
    \item algebraic cycles of dimension less than or equal to $7$, for all $g$ and $n$.
\end{enumerate}
The Tate Conjecture holds on $\Mbar_{g,n}$, over any finitely generated field, in the same range of cases.
\end{Theorem}

\noindent The theorem is proved by showing that the relevant cohomology groups are generated by tautological classes.  This method of proof is inherently limited because in many cases of higher dimension and codimension, e.g., for algebraic cycles of codimension 11 on $\Mbar_{2,20}$ or $\Mbar_{12}$, the tautological cycle classes are a proper subspace of the algebraic cycle classes \cite{GraberPandharipande03, vanZelm18}. We do not know of any cases where the Hodge conjecture is known for $H^k(\Mbar_{g,n})$ and $H^k(\Mbar_{g,n})$ has nontautological algebraic cycle classes. Some interesting open cases are highlighted in  Section~\ref{sec:outlook}. In particular, the Hodge and Tate conjectures predict algebraic cycles on spaces such as $\Mbar_{1,17} \times \Mbar_{2,14}$ and $\Mbar_{3,16}$, arising from Saito--Kurokawa and Ikeda lifts of automorphic forms, for which no geometric construction is currently known.

The Hodge conjecture and Tate conjecture pertain to the subspaces of even cohomology groups spanned by algebraic cycle classes. They have generalizations which apply to subspaces of cohomology groups, in both even and odd degrees, which are supported on an algebraic subset of specified codimension.  The classical Hodge and Tate conjectures are the respective special cases where the cohomological degree is twice the codimension of the algebraic subset. Fix $d := 3g-3+n$.

\begin{Theorem} \label{thm:b}
The generalized Hodge conjecture holds for $H^k(\Mbar_{g,n})$, as follows:
\begin{enumerate}
    \item for even $k \leq 4$, for all $g$ and $n$;
    \item for $k = 6$, for $g \not \in \{7, 8, 9\}$, and for all $n$;
    \item for odd $k \leq 13$, for all $g$ and $n$;
    \item for $k = 2d - j$, for $0 \leq j \leq \min\{ 15, 2d\}$, and for all $g$ and $n$.
\end{enumerate}
\end{Theorem}

\begin{Theorem} \label{thm:c}
Let $L$ be a finitely generated field. We assume $L$ has either characteristic zero, or positive characteristic $p \neq \ell$ such that $p$ is ordinary for $H^{k}(\Mbar_{g,n})$. Then the generalized Tate conjecture holds for $H^k((\Mbar_{g,n})_{L^s}, \QQ_\ell)$, for $g$, $n$ and $k$ as above (Theorem~\ref{thm:b}, cases (1)--(4)).
\end{Theorem}

\noindent While Theorem~\ref{thm:a} concerns the image of the cycle class map in even degree,   Theorems~\ref{thm:b} and \ref{thm:c} concern substructures rather than individual cohomology classes; they address when Hodge and Galois substructures are supported on algebraic subsets of the predicted codimension. The conjecture statements are given in Section~\ref{sec:HodgeTate}. We also prove that the action of $\Gal(L^s/L)$ on $H^k((\Mbar_{g,n})_{L^s}, \QQ_\ell)$ is semisimple in each of the cases covered by Theorem~\ref{thm:c}. See also Remark~\ref{rmk:ordinary} for a discussion of the ordinarity condition in Theorem~\ref{thm:c}.

The even degree cases are proved by showing that these cohomology groups are generated by tautological classes. The arguments for odd degrees are similar in spirit; see the examples in Section~\ref{sec:odd}. In all of the cases covered by Theorems~\ref{thm:b}-\ref{thm:c} for odd degrees $k$, the cohomology group $H^k(\Mbar_{g,n})$, with its rational Hodge structure or Galois action after tensoring with $\QQ_\ell$, is a direct sum of Tate twists of $\s_{12} := W_{11}H^{11}(\cM_{1,11})$  and $\s_{16}:= W_{15} H^{15}(\cM_{1,15})$.  The ordinarity condition means we require $p$ is ordinary for the motivic structures $\s_{12}$ and $\s_{16}$ whenever their Tate twists appear.

\subsection{Summary table} The following table summarizes the degrees $k$ in which the generalized Hodge and Tate conjectures are known for $H^k(\Mbar_{g,n})$ for all (or almost all) values of $g$ and $n$, as well as the motivic structures (Hodge structures and Galois representations) that appear and the relevant references.  The notation for the motivic structures is:
\[
\mathsf{L} := H^2(\PP^1), \quad \mathsf{S}_{12} := W_{11}H^{11}(\cM_{1,11}), \quad \mbox{ and } \quad \s_{16} := W_{15} H^{15}(\cM_{1,15}).
\]
Each is to be viewed as a rational vector space with an associated Hodge structure on the complexification and a continuous action of the absolute Galois group $\Gal(\overline \QQ /\QQ)$ on the tensor product with $\QQ_\ell$, for each prime $\ell$. We fix $d := 3g-3+n$.

For the generalized Tate conjecture, we assume that $L$ is a finitely generated field and if $L$ has characteristic $p >0$, then $p \neq \ell$. For $k = 11$, $13$, and $2d-j$ for odd $j \leq 13$, the further restriction is that $p \nmid \tau(p)$. For $2d-15$, we also require that $p$ is ordinary for $\s_{16}$, i.e. $p$ does not divide $a_p(\Delta_{16})$, the $p$th Fourier coefficient of the weight 16 cusp form of level 1 for $\SL_2(\ZZ)$.

\begin{table}[ht]
\centering
\small
\renewcommand{\arraystretch}{1.15}
\begin{tabular}{@{}llll@{}}
\toprule
Degree $k$ & Motivic structures \ \ & References & Restrictions \\
\midrule

$1,\ 3,\ 5$
& 0
& \cite{ArbarelloCornalba98}
&\\

$2$
& $\mathsf{L}$
& \cite{ArbarelloCornalba98}
& \\

$4$
& $\mathsf{L}^{2}$
& \cite{CLP24}
& \\

$6$
& $\mathsf{L}^{3}$
& \cite{CLP24}
& $g \notin \{7, 8, 9\}$ \\

$7$,\ $9$
& 0
& \cite{BFP24}
&\\

$11$
& $\mathsf{S}_{12}$
& \cite{CanningLarsonPayne23}
& $p \nmid \tau(p)$ \\

$13$
& $\mathsf{L}\mathsf{S}_{12}$
& \cite{CLPW24}
& $p \nmid \tau(p)$ \\

$2d-j$ for even $j \le 14$ \ \
& $\mathsf{L}^{(2d-j)/2}$
& \cite{CLP24}
& \\

$2d-j$ for odd $j < 11$ \ \
&   0
& \cite{BFP24}
& \\

$2d-j$ for $j = 11, 13$ \ \
&   $\mathsf{L}^{(2d-j-11)/2} \, \mathsf{S}_{12}$
& \cite{CLP24}
& $p \nmid \tau(p)$\\

$2d-15$ \ \
&   $\mathsf{L}^{(2d-26)/2} \, \mathsf{S}_{12}$, \  $\mathsf{L}^{(2d-30)/2} \, \mathsf{S}_{16}$
& \cite{CLP24}
& $p\nmid \tau(p)$, $p \nmid a_p(\Delta_{16})$ \\

\bottomrule \\
\end{tabular}

\caption{Known cases of the generalized Hodge and Tate conjectures for $H^k(\overline{\mathcal{M}}_{g,n})$ and $H^k((\Mbar_{g,n})_{L^s}, \QQ_\ell)$, respectively.}
\label{tab:known-ghc-gtc}
\end{table}

Thus, the verified picture up to degree $15$ is that low even degree cohomology is Hodge–Tate, odd cohomology vanishes below degree 11, $H^{11}$ comes from the weight 12 cusp-form motive $\s_{12}$, and $H^{13}$ comes from boundary-supported Tate twists of that motive. Moreover, homology classes in even degrees less than or equal to 14 come from tautological cycles, as do cohomology classes in degrees less than or equal to 6, except possibly for $g \in \{7,8,9\}$.

\begin{remark} \label{rmk:ordinary}
 For non-ordinary primes, the generalized Tate conjecture predicts the existence of a higher codimension algebraic subset that supports the given structure.  For instance, there should be a codimension 3 algebraic subset that supports $H^{11}((\Mbar_{1,11})_{\overline \FF_2}, \QQ_\ell)$. See Problem~\ref{prob:non-ordinary}.
This ordinarity condition can be made explicit in all of the cases under consideration.

For even $k$, the motivic structures are of Tate type and all primes are ordinary.  For odd $k \leq 13$ and odd $k \geq 2d-13$, we only have Tate twists of $\s_{12}$, and ordinarity means $p \nmid \tau(p)$.  Here, $\tau$ denotes the Ramanujan $\tau$-function, i.e., $\tau(p)$ is the $p$th Fourier coefficient of the Ramanujan cusp form $\Delta$.  In degree $2d-15$, we also see Tate twists of $\s_{16}$ (but only when $g =1$ and $n \geq 15$). In these cases, ordinarity means that $p$ also does not divide $a_p(\Delta_{16})$, the $p$th Fourier coefficient of the weight $16$ cusp form of level 1 for $\SL_2(\ZZ)$.  For both $\s_{12}$ and $\s_{16}$, it is unknown whether there are infinitely many non-ordinary primes.
\end{remark}

\subsection{Structure of the article}
The article is organized as follows. Section~\ref{sec:HodgeTate} recalls the Hodge and Tate conjectures and their generalized forms, and explains why moduli spaces of stable curves are a particularly interesting testing ground for these questions. Section~\ref{sec:tautological} discusses tautological classes and the Arbarello–Cornalba inductive method, leading to the proof of Theorem~\ref{thm:a}. Section~\ref{sec:arithmetic} describes the arithmetic predictions that motivate much of the story, especially those arising from the work of Chenevier and Lannes. Section~\ref{sec:further-results} then proves Theorem~\ref{thm:b} and discusses related conjectures on Chow groups and cycle class maps. Section~\ref{sec:outlook} concludes with a selection of open problems and future directions.

\bigskip

\noindent \textbf{Acknowledgments.}
I thank J.~Bergstr\"om, S.~Canning, C.~Faber,  H.~Larson, and J. Milne for helpful comments on a draft of this article. I am also grateful to Q.~Yin for illuminating discussions on kernels of cycle class maps and the referee for many useful suggestions.

\section{Hodge and Tate conjectures} \label{sec:HodgeTate}

We now recall the statements of the conjectures in question and discuss how and why the moduli spaces of stable curves are an interesting test case.

\subsection{The Hodge conjecture}

Let $X$ be a smooth projective variety, or a smooth and proper Deligne--Mumford stack with projective coarse moduli space, over $\CC$.  The rational singular cohomology group $H^k(X) := H^k(X(\CC),\QQ)$ carries a canonical pure Hodge
structure, i.e., a decomposition
\[
H^k(X) \otimes \CC \simeq \bigoplus_{p + q = k} H^{p,q}(X)
\]
such that $H^{p,q}(X) = \overline{H^{q,p}(X)}$. This decomposition is functorial for maps between cohomology groups induced by algebraic morphisms.

Let $Z^k(X)$ denote the free abelian group on closed irreducible algebraic subsets of codimension $k$ in $X$. Elements of $Z^k(X)$ are codimension $k$ algebraic cycles.
\begin{HC}
The rational cycle class map
\[
\cyc^k \colon Z^k(X) \otimes \QQ \to H^{2k}(X) \cap H^{k,k}(X)
\]
is surjective for all $k$.
\end{HC}
\noindent The Hodge conjecture is known in only a few important cases. First, the conjecture is true for codimension $k = 1$; this is the Lefschetz $(1,1)$-theorem. Using the hard Lefschetz theorem, it follows that the Hodge conjecture holds when $\dim(X) \leq 3$.  Also, by intersection theory, the image of $\bigoplus_k \cyc^k$ is a subring of $H^*(X)$.  In particular, the Hodge conjecture holds for $X$ whenever the subring $\bigoplus_k \big(H^{2k}(X) \cap H^{k,k}(X)\big)$ is generated by $H^{2}(X) \cap H^{1,1}(X)$. This includes very general abelian varieties in each dimension \cite{Mattuck58}. The Hodge conjecture also holds when $X$ is cellular (i.e. paved by locally closed subvarieties isomorphic to affine spaces), a class that includes Grassmannians, flag varieties, and smooth projective toric varieties \cite[Example~19.1.11]{IT}, and more generally for smooth and proper linear varieties, in the sense of Totaro \cite[Theorem~3]{Totaro14}, such as $\Mbar_{0,n}$, and for smooth and proper Deligne--Mumford stacks with the Chow--K\"unneth generation property \cite[Lemma~3.11]{CanningLarson22}. In addition, a recent preprint of Markman proves the Hodge conjecture for abelian varieties of dimension at most 5 \cite{Markman25}.

\bigskip

\subsection{The Tate conjecture}
Let $Y$ be a smooth projective variety, or a smooth and proper Deligne--Mumford stack with projective coarse moduli space, over a finitely generated field $L$.  For instance, $L$ might be a number field, a finite field, or the function field of a variety over a finite field.  The Galois group of $L$ is the automorphism group of its separable closure $$\Gal := \Gal(L^{\mathrm{s}} / L ).$$ Let $\ell$ be a prime invertible in $L$, i.e., $\ell \neq \mathrm{char}(L)$.

\begin{TC}
The $\ell$-adic rational cycle class map
\[
\cyc^k_\ell \colon Z^k(Y) \otimes \QQ_\ell \to H^{2k}(Y_{L^s}, \QQ_\ell(k))^{\Gal}
\]
is surjective for all $k$.
\end{TC}
\noindent Like the Hodge conjecture, the Tate conjecture is wide open in general, and known in only a few important cases. In particular, it is known for $k =1$ when $Y$ is an abelian variety \cite{Tate66, Faltings83, Zarhin14}, and when $Y$ is a $K3$ surface  \cite{MadapusiPera15, KimMadapusiPera16, MadapusiPera20}.

The Hodge and Tate conjectures are two of the key conjectures about algebraic cycles and motivic structures on cohomology tied to the expected properties of categories of motives as envisioned by Grothendieck. See \cite{Andre04} for a comprehensive introduction to this topic, as well as the more recent survey \cite{Totaro17} highlighting the role of the Tate conjecture.

One conjecture closely linked to the Tate conjecture is the expectation that the Galois representations arising in the cohomology of smooth projective varieties should decompose as direct sums of irreducible representations.

\begin{SSC}
The Galois action on $H^*(Y_{L^s}, \QQ_\ell)$ is semisimple.
\end{SSC}

\noindent Other closely related conjectures include the Grothendieck standard conjectures. There are many known implications among various cases of these conjectures. For instance, when $L$ has characteristic zero, the Tate conjecture implies the semisimplicity conjecture \cite{Moonen19}. Also, the Hodge conjecture for CM abelian varieties implies the Tate conjecture and the Grothendieck standard conjectures for abelian varieties over finite fields \cite{Milne99, Milne02}. Furthermore, the Grothendieck standard conjectures follow from the Hodge conjecture in characteristic zero \cite{Kleiman94}.

\medskip

For the applications in this article, the most relevant strengthenings are the generalized Hodge and Tate conjectures. Rather than focusing only on when even degree classes arise from algebraic cycles, they ask when a given substructure in cohomology must be supported on a closed subset of specified codimension. We begin with the Hodge-theoretic formulation.

\begin{GHC}
    Let $X$ be a smooth projective algebraic variety, or a smooth and proper Deligne--Mumford stack with projective coarse moduli space, over $\CC$. Let $V \subset H^j(X)$ be a rational sub-Hodge structure, with $V \otimes \CC = \bigoplus_{p+q = j} V^{p,q}$. If $V^{p,q} = 0$ for $p <  k$ then there is a closed algebraic subset $Z \subset X$ of codimension $k$ such that
    \[
V \subset \ker (H^j(X) \to H^j(X \smallsetminus Z)).
    \]
\end{GHC}
\noindent In other words, a rational sub-Hodge structure of Hodge coniveau $k$ should be supported on an algebraic subset of codimension $k$. Here, the coniveau of a substructure is a measure of the largest codimension in which it is expected to be geometrically supported; a Hodge structure has coniveau $\geq c$ exactly when $H^{p,q}$ vanishes for $p$ or $q$ less than $c$. The Hodge conjecture is recovered by taking $j = 2k$ and $V$ to be the sub-Hodge structure $H^{2k}(X) \cap H^{k,k}$.

There is a parallel formulation on the arithmetic side, where the Hodge coniveau is replaced by an analogous condition on the divisibility of Frobenius eigenvalues. The predicted geometric conclusion remains the same.

\begin{GTC}
Let $X$ be a smooth projective variety, or a smooth and proper Deligne--Mumford stack with projective coarse moduli space over a finitely generated field $L$.  Let $V \subset H^j(X_{L^s}, \QQ_\ell)$ be a Galois subrepresentation, and suppose that $V(k)$ is effective. Then there is a closed algebraic subset $Z \subset X$ of codimension $k$ such that
\[
V \subset \ker \big(H^j(X_{L^s}, \QQ_\ell) \to H^j((X \smallsetminus Z)_{L^s}, \QQ_\ell)\big).
\]
\end{GTC}

\noindent Here we consider smooth proper models of $X$ over normal integral schemes of finite type over $\ZZ[1/\ell]$ with function field $L$. We say that $V(k)$ is effective if its Frobenius eigenvalues are algebraic integers at every closed point of every such base.

\noindent In other words, a sub-Galois representation of Tate coniveau $k$ should be supported on an algebraic subset of codimension $k$. The Tate conjecture is recovered by taking $j = 2k$ and $V \subset H^{2k}(X_{L^s}, \QQ_\ell)$ to be the untwisted form of the invariant subspace $H^{2k}(X_{L^s}, \QQ_\ell(k))^{\Gal}$.

These generalized formulations, together with the semisimplicity conjecture for the $\ell$-adic Galois representations, are the natural framework for the results discussed below. They turn questions about motivic structures into concrete geometric statements. This perspective is well-suited to $\Mbar_{g,n}$ where, in many cases, the boundary strata, together with cycles representing the $\kappa$- and $\psi$-classes, and intersection products thereof, supply natural candidates for the expected support loci.

\subsection{Moduli spaces of stable curves as a test case}

The moduli spaces of stable algebraic curves $\Mbar_{g,n}$, for $2g - 2 + n > 0$, occupy a central position in modern mathematics, connecting algebraic geometry to other branches of mathematics ranging from combinatorics and representation theory to number theory and mathematical physics. Mumford famously detailed his obsession with these spaces, which he described as a ``passion flower" in the ``secret garden" of algebraic geometry \cite{Mumford97}.

The structure of these spaces, and the natural morphisms between them, are extraordinarily deep and intricate. The combinatorics of the boundary $$D_{g,n} := \Mbar_{g,n} \smallsetminus \cM_{g,n}$$ is encoded in the moduli space of tropical curves \cite{acp}; its structure is sufficiently intricate to encode Kontsevich's commutative graph complex and, in particular, the Grothendieck--Teichm\"uller Lie algebra \cite{CGP21}.  Meanwhile, the pure weight cohomology of the interior $\cM_{g,n}$ is a rich source of motives and Galois representations over $\ZZ$ \cite{FvdG04, FvdG04b, BvdG08, BFvdG14, BergstromFaber23}.

It is of fundamental importance that both $\Mbar_{g,n}$ and its normalized boundary $\widetilde D_{g,n}$ are smooth and proper over $\Spec \ZZ$ \cite{DeligneMumford69, Knudsen83}. The motives that can occur in the cohomology of such spaces are tightly constrained by arithmetic. Indeed, the work of Chenevier and Lannes \cite{ChenevierLannes19} classifies the level-one algebraic cuspidal automorphic representations of motivic weight at most $22$. Assuming standard conjectures relating motives and automorphic forms, this yields a short, explicit list of the motives of weight at most $22$ that can occur in the cohomology of \emph{any} space smooth and proper over $\Spec \ZZ$. Remarkably, every motive on this list appears in $H^*(\Mbar_{g,n})$ for suitable $g$ and $n$. We take up this classification and its relations to cohomology in Section~\ref{sec:arithmetic}.

Recent geometric work \cite{CGP21, CGP22, PayneWillwacher24, PayneWillwacher24b, CLPW25, CLPW25b} casts light on the patterns with which each motive appears in the cohomology of both $\cM_{g,n}$ and $\Mbar_{g,n}$, as $g$ and $n$ vary.  Up to weight 15, the appearance of each small motive in $\cM_{g,n}$ is governed by an explicit graph complex, and techniques for computing with graph complexes shed new light on the cohomology of these moduli spaces. For instance, we now know that the cohomology of $\cM_{g,n}$ is of Tate type if and only if $g = 0$ or $3g + 2n < 25$ \cite{CLPW24}.

Since we have a precise list of the expected motives in each weight up to $22$, and since we have geometric methods for computing where each of these motives appears, the cohomology of $\Mbar_{g,n}$ is a rich testing ground for general conjectures about motives, Hodge structures, and Galois representations.

\medskip

Throughout, whenever we refer to $\Mbar_{g,n}$, we assume that $2g -2 + n > 0$.

\begin{remark}\label{rmk:effectivity}
The effectivity condition for $\s_{12}$ and $\s_{16}$ can be checked directly. At a prime $p \neq \ell$, their Frobenius polynomials are
\[
T^2-\tau(p)T+p^{11}
\qquad\text{and}\qquad
T^2-a_p(\Delta_{16})T+p^{15},
\]
respectively. Twisting by $1$ divides the Frobenius eigenvalues by $p$, so their sum becomes $\tau(p)/p$ for $\s_{12}$ and $a_p(\Delta_{16})/p$ for $\s_{16}$. At an ordinary prime, these sums are not integers. Since the Weil bounds show that each pair consists of nonreal conjugate roots, neither eigenvalue after twisting can be an algebraic integer. It follows that $V(1)$ is not effective for any nonzero subrepresentation $V$. The same is true over $\FF_{p^m}$, where the eigenvalues are replaced by their $m$th powers.

Both $17$ and $19$ are ordinary for $\s_{12}$ and $\s_{16}$ \cite[p.~203]{Gouvea97}, so one can always choose such a prime different from $\ell$. For a finitely generated field $L$ of characteristic zero, choose a normal integral model proper over $\Spec\ZZ[1/\ell]$. Its fibre over the chosen prime is nonempty. Since $\Mbar_{g,n}$ is smooth and proper over $\ZZ$, a closed point of this fibre gives a good reduction over a finite extension of $\FF_p$, and the preceding argument applies. Thus no nonzero subrepresentation of $\s_{12}$ or $\s_{16}$ has positive Tate coniveau in characteristic zero. In characteristic $p$, the same conclusion holds when $p$ is ordinary.
\end{remark}

\subsection{Odd examples} \label{sec:odd} For odd $k < 11$, we have $H^k(\Mbar_{g,n}) = 0$ for all $g$ and $n$ \cite{BFP24}. The proof of this topological fact relies heavily on algebraic geometry (point counting over finite fields and the Behrend-Grothendieck-Lefschetz trace formula for algebraic stacks). Given this vanishing, the generalized Hodge and Tate conjectures hold vacuously in these cases.

For $k =11$, we have the famous example $H^{11}(\Mbar_{1,11}) \cong \QQ^2$, with Hodge type
\[
H^{11}(\Mbar_{1,11}, \CC) \cong H^{11,0}\oplus H^{0,11}.
\]
In particular, $H^{11}(\Mbar_{1,11}, \CC)$ is spanned by a holomorphic form and its complex conjugate. This holomorphic form has origins in arithmetic, and comes, in particular, from the Ramanujan cusp form of weight $12$ for $\SL_2(\ZZ)$, as we now recall.

By mixed Hodge theory, the subspace $W_kH^k(\cM_{g,n})$ is the image of $H^k(\Mbar_{g,n})$ under restriction to the open subspace parameterizing smooth curves. It carries a pure Hodge structure of weight $k$, given by $W_k H^k(\cM_{g,n}) \otimes \CC = \bigoplus_{p + q = k} H^{p,q}$. For $g = 1$ and $n = k$, the Eichler--Shimura correspondence identifies the $(k,0)$-summand with the space of cuspidal modular forms of weight $k + 1$ and level $1$ for $\SL_2(\ZZ)$. This can be made completely explicit, as in \cite[Section~2.3]{FaberPandharipande13}. Let $\HH$ denote the upper half-plane, the universal cover of $\cM_{1,1}$. Then $\HH \times \CC^{k-1}$ is the universal cover of the $(k-1)$-fold fiber product of the universal curve, which contains $\cM_{1,k}$ as a dense open subspace. The transformation rule for a modular form $f$ of weight $k+1$ is exactly such that the restriction of the holomorphic form $f dz \wedge dz_1 \wedge \cdots \wedge dz_{k-1}$ on $\HH \times \CC^{k-1}$ to the preimage of this open set descends to a holomorphic form on $\cM_{1,k}$. The cuspidal condition ensures that this form extends to $\Mbar_{1,k}$.  In particular, the image of the Ramanujan cusp form $\Delta$ under this correspondence spans $H^{11,0}(\Mbar_{1,11}) \cong \CC$. Since $H^{11}(\Mbar_{1,11})$ is irreducible as a rational Hodge structure, and $H^{11,0} \neq 0$, there is no nontrivial substructure with positive Hodge coniveau, and hence the generalized Hodge conjecture is true in this case. Similarly, under the stated hypothesis on the characteristic, no nonzero sub-Galois representation of $H^{11}((\Mbar_{1,11})_{L^s}, \QQ_\ell)$ has positive Tate coniveau, by Remark~\ref{rmk:effectivity}. Hence the generalized Tate conjecture holds in this case. Note that this explicit description of the holomorphic 11-form on $\Mbar_{1,11}$ shows that the symmetric group $S_{11}$ acts via the sign representation; similar computations show that the holomorphic forms of positive degree on $\Mbar_{1,n}$ are never $S_n$-invariant, a pleasant observation that also follows from the fact that the coarse space of $\Mbar_{1,n}/S_n$ is unirational for all $n$ \cite[Proposition~2.2]{FarkasVerra13}.

Recent results show that $H^{11}(\Mbar_{g,n})$ vanishes unless $g = 1$ and $n \geq 11$, in which case it is a direct sum of pullbacks of $H^{11}(\Mbar_{1,11})$ \cite{CanningLarsonPayne23}. In particular, any nonzero rational sub-Hodge structure has non-vanishing $H^{11,0}$, and, under the hypotheses of Theorem~\ref{thm:c}, the Tate twist $V(1)$ is not effective for any nonzero sub-Galois representation $V \subset H^{11}((\Mbar_{g,n})_{L^s}, \QQ_\ell)$.  The generalized Hodge and Tate conjectures therefore hold for $k = 11$. These holomorphic forms are also the source of the ``huge Chow groups'' mentioned in the introduction: their presence makes the rational Chow groups of $0$-cycles on $\Mbar_{1,n}$ vastly larger than the tautological subring for $n \geq 11$. In Section~\ref{sec:further-results}, we discuss the conjecture that the cycle class map is an isomorphism in codimension at most 5, and the expectation that its kernel can be infinite-dimensional in codimension 6.

The first case where there is a nontrivial substructure of positive Hodge coniveau is for $H^{13}(\Mbar_{1,12})$. Explicit computation shows that this cohomology group is isomorphic to $\QQ^{22}$, with Hodge types $$H^{13}(\Mbar_{1,12}, \CC) \cong H^{12,1} \oplus H^{1,12}.$$ Moreover, the symmetric group $S_{12}$ acts on $H^{12,1}$ as the irreducible representation associated to the hook shape partition $(2,1^{10})$ (see, e.g., \cite{BergstromData}).  The generalized Hodge conjecture predicts that, since $H^{13,0} = 0$, the group $H^{13}(\Mbar_{1,12})$ should be supported on a divisor.

For each subset $A \subset \{1, \ldots, 12\}$ of size $10$, there is a gluing map
\[
\iota_A \colon \Mbar_{1, A \cup \{p \}} \times \Mbar_{0, A^c \cup \{q\}} \to \Mbar_{1, 12},
\]
obtained by gluing $p$ to $q$.  The domain is isomorphic to $\Mbar_{1,11}$, and the Gysin push-forward of $H^{11}(\Mbar_{1,11})$ is nonzero; this can be seen by pulling back under other tautological morphisms and taking K\"unneth components. Using the Hodge structure and irreducibility of $H^{12,1}(\Mbar_{1,12})$ as an $S_{12}$-representation, we conclude that the sum of the Gysin maps $\bigoplus_A \iota_{A*}$ surjects onto $H^{13}(\Mbar_{1,12})$. Hence, $H^{13}(\Mbar_{1,12})$ is supported on the union of the images of these gluing maps. In particular, it is supported on a divisor, as predicted by the generalized Hodge conjecture.

\begin{remark}
    The fact that $H^{13}(\Mbar_{1,12})$ is supported on a divisor was proved decades ago \cite[Proposition~6]{GraberPandharipande03}.  The generalization to arbitrary $g$ and $n$ is more recent \cite[Section~4]{CLPW24}.
\end{remark}

\subsection{Motivic structures}

The topics discussed here are informed by the philosophy of motives, which predicts a geometrically defined abelian category with properties that explain and unify observed phenomena in Hodge theory, algebraic cycles, and point counts over finite fields. However, lacking proof that any geometrically defined category has all of the expected properties, we work naively, and define a motivic structure of weight $k$ to be a rational vector space $V$ together with a pure Hodge structure of weight $k$ on $V \otimes \CC$ and a continuous $\ell$-adic Galois representation of weight $k$ on $V \otimes \QQ_\ell$, with suitable compatibilities and comparison isomorphisms \cite{Deligne89}. (In other contexts, one may also wish to include other realizations, such as de Rham and crystalline cohomology, but for simplicity we leave these aside.)

\section{Tautologically generated cohomology groups} \label{sec:tautological}

The tautological subring $RH^*(\Mbar_{g,n}) \subset H^*(\Mbar_{g,n})$ is the $\QQ$-subalgebra generated by classes that arise most naturally from the projection map from the universal curve and the stratification of the boundary. More precisely, it is the $\QQ$-subalgebra generated by the $\kappa$- and $\psi$-decorated boundary stratum classes. Equivalently, $\{ RH^*(\Mbar_{g,n}) \}_{g,n}$ is the smallest system of subrings that contains $\QQ$ and is closed under push-forward for all tautological morphisms.  For each $k$, we have a chain of  inclusions
\begin{equation} \label{eq:nested}
  RH^{2k}(\Mbar_{g,n}) \;\subseteq\; H^{2k}_{\mathrm{alg}}
  \;\subseteq\; H^{2k}(\Mbar_{g,n})\cap H^{k,k} \subseteq H^{2k}(\Mbar_{g,n}),
\end{equation}
where $H^{2k}_{\mathrm{alg}}$ is the image of $\cyc^k$ and $RH^{2k}$ is the image of the tautological cycles.

One says that $H^k(\Mbar_{g,n})$ is tautological when $RH^k(\Mbar_{g,n}) = H^k(\Mbar_{g,n})$. In particular, if $k$ is odd then $H^k(\Mbar_{g,n})$ is tautological if and only if it is zero. Note that $H^{2k}$ being tautological is the statement that all of the inclusions in \eqref{eq:nested} are equalities. The Hodge conjecture is the weaker statement that $H^{2k}_{\mathrm{alg}} = H^{2k}(\Mbar_{g,n})\cap H^{k,k}$, and examples show that the other two inclusions can be strict.

The systematic study of these tautological rings was initiated by Mumford, in pursuit of an enumerative geometry of curves analogous to Schubert calculus on Grassmannians and flag varieties \cite{Mumford83}.
It is now the most intensively studied part of the intersection theory of moduli spaces of curves, particularly in the context of enumerative geometry and Gromov-Witten theory \cite{FaberPandharipande00, FaberPandharipande03,FaberPandharipande05,  FaberPandharipande13, Pixton13, PPZ15}.

The full cohomology ring $H^*(\Mbar_{g,n})$ is tautological in a range of small cases. This includes for $g = 0$, for $g = 1$ and $n \leq 10$, for $g = 2$ and $n \leq 9$, and for $g \geq 3$ and $2g + n \leq 14$ \cite{Keel92, Petersen14, CanningLarson22}.  Odd cohomology first arises in $H^{11}(\Mbar_{1,11})$. However, the even degree cohomology of $\Mbar_{1,n}$ is tautological for all $n$ \cite{Petersen14}. Non-tautological classes also appear in even degree, and even in $H^{2k}(\Mbar_{g,n}) \cap H^{k,k}$. This was first observed on $\Mbar_{2,20}$ \cite{GraberPandharipande03}. The first example without marked points comes from the closure of the bielliptic locus in $\Mbar_{12}$ \cite{vanZelm18}. Many more examples are given in  \cite{ACCHLMT25}.

A related problem is to determine the degrees $k$ such that $H^k(\Mbar_{g,n})$ is tautological for all $g$ and $n$.  Arbarello and Cornalba proved that $H^2(\Mbar_{g,n})$ is tautological for all $g$ and $n$ \cite{ArbarelloCornalba98}. More recently, it was shown that $H^4(\Mbar_{g,n})$ is tautological for all $g$ and $n$, and that $H^6(\Mbar_{g,n})$ is tautological for all $g \notin \{7,8,9\}$ and all $n$ \cite{CLP24}.

\begin{proposition} \label{prop:tautological}
If $H^{2k}(\Mbar_{g,n})$ is tautological then the Hodge conjecture holds for algebraic cycles of codimension $k$ on $\Mbar_{g,n}$, as does the Tate conjecture for algebraic cycles of codimension $k$ over any finitely generated field $L$.
\end{proposition}

\begin{proof}
Suppose the rational singular cohomology group $H^{2k}(\Mbar_{g,n})$ is generated by tautological algebraic cycle classes. For the Hodge conjecture, the proposition is clear.  We now show that the Tate conjecture holds for algebraic cycles of codimension $k$ on $\Mbar_{g,n}$ over any finitely generated field $L$, using the standard comparison theorems for schemes and Deligne--Mumford stacks \cite{Milne80, Behrend04}.

If $L$ has characteristic zero, then the choice of an embedding $L^s \hookrightarrow \CC$ determines a comparison isomorphism
$$H^{2k}((\Mbar_{g,n})_{L^s}, \QQ_\ell(k)) \xrightarrow{\sim} H^{2k}(\Mbar_{g,n}(\CC), \QQ(k)) \otimes \QQ_\ell,$$ which is compatible with the cycle class maps. The tautological classes are defined over $\QQ$ and hence invariant under $\Gal(L^s/L)$. Thus, if tautological cycle classes generate $H^{2k}(\Mbar_{g,n}(\CC), \QQ)$, then they also generate $H^{2k}((\Mbar_{g,n})_{L^s}, \QQ_\ell(k))$, and the Tate conjecture holds in these cases.

It remains to consider the case where $L$ has characteristic $p > 0$.  Since the structure map $f \colon \Mbar_{g,n} \to \Spec \ZZ$ is smooth and proper, the sheaf $R^{2k} f_* \QQ_\ell(k)$ is lisse over $\Spec \ZZ[\frac{1}{\ell}]$. Since the tautological cycle classes are defined over $\ZZ$, they give global sections of this sheaf. If these sections span the characteristic zero fiber, then they span every geometric fiber over $\Spec \ZZ[\frac{1}{\ell}]$. In particular, their span contains the Galois-invariant subspace.
\end{proof}

We briefly recall the inductive method of Arbarello and Cornalba \cite{ArbarelloCornalba98}, which underlies much of the recent progress in this area.  Recall that
\[
D_{g,n}:=\overline{\mathcal M}_{g,n}\setminus \mathcal M_{g,n},
\]
denotes the boundary divisor in the moduli space of stable curves, and $\widetilde D_{g,n}$ is its normalization, which is a disjoint union of products of smaller moduli spaces (or quotients of such products by finite groups). The localization sequence for the open inclusion $\cM_{g,n} \subset \Mbar_{g,n}$, together with the yoga of weights, gives a right exact sequence
\begin{equation} \label{eq:exc1}
H^{k-2}(\widetilde D_{g,n}) (-1) \to H^k(\Mbar_{g,n}) \to W_k H^k(\cM_{g,n}) \to 0.
\end{equation}
In this way, every class in $H^k(\Mbar_{g,n})$ is built out of cohomology classes on smaller moduli spaces of curves and classes in $W_k H^k(\cM_{g,n})$, the pure part of the cohomology of the interior.

There are many techniques for understanding the cohomology groups $H^k(\cM_{g,n})$ and their pure weight subspace $W_k H^k$.  Some of these techniques are topological, drawing on nontrivial inputs from mapping class groups and from stable homotopy theory  \cite{Harer85, Harer86, MadsenWeiss07}. In particular, Harer's computation of the virtual cohomological dimension of mapping class groups of surfaces gives strong vanishing results for $H^k(\cM_{g,n})$ when $k$ is large relative to $g$ and $n$. The relevant bound is roughly $4g-4+n$, for $g \geq 2$. Equivalently, by Poincar\'e duality, the compactly supported cohomology of $\cM_{g,n}$ vanishes in degrees less than roughly $2g + n - 1$. We then have the alternative formulation of excision of the boundary,
\begin{equation} \label{eq:exc2}
\cdots \to H^k_c(\cM_{g,n}) \to H^k(\Mbar_{g,n}) \to H^k(D_{g,n}) \to \cdots
\end{equation}
In particular, when $k$ is small, the first term in \eqref{eq:exc2} vanishes and $H^k(\Mbar_{g,n})$ injects into $H^k(D_{g,n})$. The kernel of pullback from $H^k(D_{g,n})$ to $H^k(\tilde D_{g,n})$ has weights below $k$. Since $H^k(\Mbar_{g,n})$ is pure of weight $k$ and every morphism of mixed Hodge structures is strict, the composition with pullback to the normalization is again injective \cite[Lemma~2.6]{ArbarelloCornalba98}. Thus, when $k$ is small there is an injection $H^k(\Mbar_{g,n}) \hookrightarrow H^k(\tilde D_{g,n})$, so every class in $H^k(\Mbar_{g,n})$ is determined by its restrictions to smaller moduli spaces.

Another key input from topology is stabilization: when $k$ is smaller than $\frac{2g-2}{3}$, the cohomology of $\cM_{g,n}$ is stable and freely generated by tautological $\kappa$- and $\psi$-classes. For $\cM_g$, this is the famous theorem of Madsen and Weiss \cite{MadsenWeiss07} conjectured by Mumford \cite{Mumford83}. See \cite{Looijenga96} for the extension to $\cM_{g,n}$.

Elaborations on these ideas, combining nontrivial inputs from topology with techniques from mixed Hodge theory, are at the heart of the proofs in \cite{ArbarelloCornalba98}, and continue to form the basis of many inductive arguments about the cohomology of $\Mbar_{g,n}$.

The recent work \cite{CLP24} follows the same philosophy but packages it in a form better suited to controlling high degree cohomology groups. Rather than working with either the full cohomology ring, or only with the tautological subring, one allows intermediate systems of subrings, formalized as semi-tautological extensions, that are still closed under the tautological operations that move classes from smaller moduli spaces to larger moduli spaces, by push-forward for gluing and pullback for forgetting markings. The localization sequence \eqref{eq:exc1} first accounts for the boundary classes, which restrict to zero on the interior. It remains to control the quotient of $W_kH^k(\mathcal M_{g,n})$ by tautological classes and pullbacks from fewer markings. A Leray--K\"unneth analysis of the universal curve shows that this quotient vanishes once $n$ is large relative to $k$. Homological stability gives an upper bound on $g$, and virtual cohomological dimension eliminates further cases, leaving only finitely many pairs $(g,n)$ for each $k$.

To get strong new results from this formalism, one also needs new ways of establishing the base cases needed to run the inductive arguments in this framework. In \cite{CLP24}, this is done using explicit low-genus calculations and algebraicity results for pure weight cohomology, via the Chow--K\"unneth generation property. The Chow--K\"unneth generation property is established geometrically, on a case-by-case basis. For instance, for certain moduli of genus $7$ curves, the key geometric input comes from Mukai's description of canonical curves as linear sections of an orthogonal Grassmannian.
With this framework in mind, the proof of Theorem~A is essentially an accounting of the base cases that are currently accessible.

\begin{proof}[{\bf Sketch of proof of Theorem~\ref{thm:a}}]
Arbarello and Cornalba computed $H^2(\Mbar_{g,n})$ for all $g$ and $n$, showing that it is generated by the tautological classes $\kappa_1, \psi_1, \ldots, \psi_n$ and boundary divisors \cite{ArbarelloCornalba98}.
By Proposition~\ref{prop:tautological}, this proves the theorem for cycles of codimension 1.

For cycles of codimension 2, we follow a similar argument, namely a double induction on $g$ and $n$ using the weight spectral sequence for the open inclusion $\cM_{g,n} \subset \Mbar_{g,n}$. The argument also relies on knowledge of the stable cohomology of the open moduli spaces $\cM_{g,n}$ in low degrees, and vanishing theorems for their high degree cohomology \cite{Harer85, Harer86, Looijenga96, MadsenWeiss07}. The details are given in \cite{CLP24} and show that $H^4(\Mbar_{g,n})$ is generated by tautological cycle classes for all $g$ and $n$.

In \cite{CLP24}, we observed that an easy variation on these arguments also shows that $H^6(\Mbar_{g,n})$ is tautological for $g \geq 10$, and for all $n$. This extends to $g \leq 6$, by a closer analysis of the base cases, as follows. By \cite[Remark~6.1]{CLP24}, the base cases required to show that $H^6(\Mbar_{g,n})$ is tautological for all $g$ and $n$ are the cases where $g \leq 9$ and $n \leq 6$. Then, \cite[Lemma~3.1(b)]{CLP24} reduces this further to $n \leq 5$, since $W_6 H^6(\cM_{g,n})$ is pulled back from spaces with fewer marked points once $n \geq 6$. As $H^4$ is tautological for all $g$ and $n$, the classes in $H^6(\Mbar_{g,n})$ pushed forward from the boundary are tautological, and \eqref{eq:exc1} shows that $H^6(\Mbar_{g,n})$ is tautological if and only if the pure weight cohomology $W_6 H^6(\cM_{g,n})$ is tautological. For $g \leq 2$, this holds because $H^*(\Mbar_{g,n})$ is itself tautological in this range \cite{Petersen14, CanningLarson22}. For $3 \leq g \leq 6$ and $n \leq 5$, the open moduli space $\cM_{g,n}$ has tautological rational Chow ring and satisfies the Chow--K\"unneth generation property \cite[Theorem~1.4]{CanningLarson22}. It follows that $W_6 H^6(\cM_{g,n})$ is algebraic \cite[Lemma~4.3]{CLP24}, and these algebraic classes are tautological. Running the induction therefore shows that $H^6(\Mbar_{g,n})$ is
tautological for all $g \leq 6$ and all $n$. The cases $g \in \{7,8,9\}$ remain open.

The tautological ring contains all divisor classes, including an ample class. Therefore, from what we have discussed above and the hard Lefschetz theorem, we know that $H^{2d-k}(\Mbar_{g,n})$ is generated by the classes of tautological algebraic cycles for even $k \leq 4$, as well as for $k = 6$ and $g \not \in \{7, 8, 9 \}$.  However, in \cite[Theorem~1.5]{CLP24}, we also show that $H^{2d-k}(\Mbar_{g,n})$ is generated by the classes of tautological algebraic cycles, for even $k \leq 14$.  This proves the theorem for algebraic cycles of dimension less than or equal to $7$, for all $g$ and $n$.
\end{proof}

\begin{remark}
As noted above, one can deduce tautological generation and related results in high degree $H^{2d-k}$ from the corresponding results for $H^k$, when $k<d$, using the hard Lefschetz theorem.  Thus, the high degree results are a priori easier than the corresponding low degree results. Nevertheless, the extent to which one obtains stronger results for cycles of low dimension rather than low codimension is striking. We note that this distinction lessens if one assumes the Grothendieck standard conjectures. In particular, the Lefschetz standard conjecture predicts that the inverse of the hard Lefschetz isomorphism is induced by an algebraic correspondence; assuming this conjecture (which is wide open in general), the Hodge conjecture for $H^{2d-k}$ implies the Hodge conjecture for $H^k$, when $k < d$.

We note that the methods of proof for the high degree and low degree results stated here are similar: one rigorously applies the method of Arbarello and Cornalba and establishes as many base cases as possible with current methods. The recent progress required to prove Theorem~\ref{thm:a} relies most heavily on new arguments to establish the base cases for relatively small values of $g$ and $n$ needed to run the induction. This includes point counting arguments \cite{BFP24}, using the Behrend--Grothendieck--Lefschetz trace formula for stacks \cite{Behrend93}, and applications of the Chow--K\"unneth generation property \cite{CanningLarson22, CLP24}, together with explicit classical constructions of small genus moduli spaces and their gonality stratifications.
\end{remark}

In \cite[Conjecture~1.8]{CLP24}, there is the following proposed extension of Theorem~\ref{thm:a}.

\begin{Taut20}
The cohomology groups $H^{k}(\Mbar_{g,n})$ are tautological for all even $k \leq 20$, and for all $g$ and $n$.
\end{Taut20}

\noindent By the hard Lefschetz theorem, it would follow that $H_k(\Mbar_{g,n})$ is also tautological for all even $k \leq 20$. The conjectured bound is tight (or \emph{taut}), since $H^{22}(\Mbar_{2,20})$ and $H^{22}(\Mbar_{12})$ are not tautological \cite{GraberPandharipande03, vanZelm18}. If we include odd degree cohomology groups, then the best possible bound drops to $10$.

\begin{Taut10}
All cohomology groups $H^k(\Mbar_{g,n})$ are tautological for $k \leq 10$.
\end{Taut10}

\noindent This conjectured bound is also best possible, since $H^{11}(\Mbar_{1,11})$ is not tautological. The taut 10 conjecture follows from the taut 20 conjecture, since $H^k(\Mbar_{g,n})$ vanishes for odd $k < 11$ \cite{BFP24}.

It is also natural to consider weaker forms of these conjectures, for low degree homology groups. Let $H_k(\Mbar_{g,n}) := H_k(\Mbar_{g,n}(\CC),\QQ)$ denote rational singular homology. The Poincar\'e pairing induces isomorphisms between rational singular homology and cohomology in complementary degrees
\begin{equation} \label{eq:Poincare-pairing}
    H_k(\Mbar_{g,n}) \cong H^{2d-k}(\Mbar_{g,n}).
\end{equation}
Let $RH_k$ denote the subspace of $H_k(\Mbar_{g,n})$ generated by tautological cohomology classes, via \eqref{eq:Poincare-pairing}. We say that $H_k(\Mbar_{g,n})$ is tautological if it is equal to $RH_k(\Mbar_{g,n})$.
By the hard Lefschetz theorem, iterated cup product with an ample class $\omega$ gives an isomorphism
\begin{equation*}
 L_\omega^{d-k} \colon H^{k}(\Mbar_{g,n}) \xrightarrow{\sim} H^{2d-k}(\Mbar_{g,n}),
\end{equation*}
for $k \leq d$. Since every divisor class is tautological, and in particular the ample class $\omega$, this isomorphism restricts to an injection
\begin{equation} \label{eq:HL-inclusion}
 RH^{k}(\Mbar_{g,n}) \hookrightarrow RH^{2d-k} (\Mbar_{g,n}) \cong RH_k(\Mbar_{g,n}) .
\end{equation}
In particular, if $H^k(\Mbar_{g,n})$ is tautological then so is $H_k(\Mbar_{g,n})$. Note, however, that the inclusion \eqref{eq:HL-inclusion} can be strict, as is the case for $(g,n) = (2,20)$ \cite{Petersen16}.

The taut 20 conjecture seems to be significantly harder than the homological variant, i.e., the conjecture that $H_k(\Mbar_{g,n})$ is tautological for even $k \leq 20$.  Indeed, it is already known that $H_k(\Mbar_{g,n})$ is tautological for even $k \leq 14$ \cite[Theorem~1.5]{CLP24}.

\section{Inspiration from arithmetic} \label{sec:arithmetic}

While the results discussed above are proved geometrically, the statements that have been investigated so far draw heavily on inspiration from arithmetic. The work of Chenevier and Lannes, classifying the automorphic representations of motivic weight less than or equal to 22 \cite[Theorem~F]{ChenevierLannes19} has been particularly influential. See also \cite{ChenevierTaibi20} for a simplified proof of this classification and a partial extension to weights 23 and 24.

The classification is proved using analytic $L$-function methods and, in particular, a variant of Odlyzko's minoration of discriminants \cite{Odlyzko75, Odlyzko90}, adapted to conductors of algebraic varieties, cf. \cite{Mestre86}.

\medskip

To state the classification in the geometric setting, we must specify a few $\ell$-adic Galois representations. For positive integers $k$, let
\[
\s_{k+1} := W_k H^k(\cM_{1,k}),
\]
which we view as a motivic structure, i.e., a rational vector space together with a pure Hodge structure on its complexification and a continuous action of $\Gal(\overline \QQ / \QQ)$ on its tensor product with $\QQ_\ell$, for each prime $\ell$.

By the Eichler--Shimura correspondence, $\s_{k+1}$ has dimension twice the dimension of the space of cusp forms of weight $k+1$ and level $1$ for $\SL_2(\ZZ)$, and the associated Hodge weights are $(k,0)$ and $(0,k)$. In particular, as $\ell$-adic representations,
\[
\s_{12}, \ \s_{16}, \ \s_{18}, \ \s_{20},  \mbox{ and } \ \s_{22}
\]
are irreducible and 2-dimensional, of motivic weights 11, 15, 17, 19, and 21, respectively.

For nonnegative integers $j$ and $k$, let $\s_{j,k}$ denote the Galois representation attached to the Siegel cusp forms of weight $\Sym^j \otimes \det^k$, which are realized in the appropriate weight-graded piece of the cohomology of a local system on $\cA_2$ \cite{Petersen15}. The motivic weight of $\s_{j,k}$ is $j + 2k -3$.
With this notation,
\[
\s_{6,8}, \ \s_{4,10}, \ \s_{8,8}, \ \mbox{ and } \ \s_{12,6}
\]
are irreducible Galois representations of motivic weights 19, 21, 21, and 21, respectively. For further details on these Galois representations and their relations to holomorphic forms on moduli spaces of stable curves, see \cite[Section~3]{CLPW25b} and the references therein.

The classification is, in the first instance, an unconditional statement about automorphic representations and their attached Galois representations. Let $$R := \{1; \  \s_{12}, \ \s_{16}, \ \s_{18}, \ \s_{20}, \  \s_{22}; \ \s_{6,8}, \ \s_{4,10}, \ \s_{8,8}, \ \s_{12,6}; \ \Sym^2 \s_{12}\}.$$

\begin{theorem}[\cite{ChenevierLannes19}] \label{thm:CL-aut}
There are exactly eleven cuspidal automorphic representations of $\mathrm{PGL}_m$ over $\QQ$, for $m \geq 1$, that are algebraic, unramified at every prime, and of motivic weight at most $22$. Their attached $\ell$-adic Galois representations are precisely the elements of $R$.
\end{theorem}

\noindent The analytic properties of Rankin--Selberg $L$-functions used in the proof are unconditional theorems when the representations involved are automorphic. Assuming those same properties hold for more general Galois representations yields the following.

\begin{theorem}[\cite{ChenevierLannes19}] \label{thm:CL}
Let $V$ be an effective irreducible $\ell$-adic Galois representation belonging to a pure compatible system of conductor $1$ and motivic weight $k \leq 22$. Assume that for each $W$ in $R \cup \{V\}$, the completed $L$-functions attached to $V \otimes W$ and $V \otimes W^\vee$ have the expected analytic continuation, functional equation, epsilon factors, zeros, and poles. Then $V$ is a Tate twist of a representation in $R$.
\end{theorem}

\noindent Applied to the cohomology of varieties and Deligne--Mumford stacks that are smooth and proper over $\ZZ$, Theorem~\ref{thm:CL} yields striking, if conditional, constraints on the motivic structures that can occur. We now discuss these predictions, which are important inspiration for the results discussed in Section~\ref{sec:further-results}, which confirm many of the low weight predictions for moduli spaces of stable curves. The proofs of these results are geometric and unconditional.

By the conjectures of Fontaine--Mazur and Langlands, every irreducible constituent $V$ of $H^k(X_{\overline\QQ}, \QQ_\ell)$, for $X$ smooth and proper over $\ZZ$, should be automorphic. Assuming the Tate conjecture, which implies that the $\ell$-adic realization functor is fully faithful on pure motives over $\QQ$ modulo homological equivalence (after tensoring with $\QQ_\ell$), together with the expected semisimplicity of these representations, one expects that $H^k(X_{\overline \QQ}, \QQ_\ell)$ is a direct sum of $\ell$-adic realizations of irreducible pure motives, each of which is a Tate twist of an element of $R$, for $k \leq 22$. In particular, the cohomology groups $H^k(\Mbar_{g,n})$ should be pure Tate for even $k \leq 20$ and should vanish for odd $k < 11$, since $R$ contains no nontrivial representations of odd weight less than 11, and the only non-trivial representation of even weight is $\Sym^2 \s_{12}$, in weight 22.

\begin{remark}
A striking feature of Theorem~\ref{thm:CL-aut} is that the predicted list of Galois representations of weight at most 22 and conductor 1 is \emph{finite}. See  \cite{ChenevierTaibi20} for a simplified proof, an extension of this finiteness result to weight 23 and, conditional on the Generalized Riemann Hypothesis, to weight 24. Chenevier, Lannes, and Ta\"ibi do not assume any \emph{a priori} bound on the rank and hence their results do not follow from the standard finiteness results for $\ell$-adic Galois representations \cite{Faltings83, Deligne85}.  The corresponding finiteness is an open problem, even assuming GRH, for weights 25 and greater.  See \cite[Section~1.5]{ChenevierTaibi20}.
\end{remark}

The predictions coming from Theorem~\ref{thm:CL} are consistent with the few previously known classification results for varieties that are smooth and proper over $\Spec \ZZ$. In particular, since positive genus curves and nontrivial abelian varieties have nontrivial $H^1$, these cohomological predictions agree with the theorem of Fontaine that there are no nontrivial abelian varieties over $\Spec \ZZ$. Consequently, there are no nontrivial curves of positive genus over $\ZZ$, and since $\mathrm{Br}(\ZZ) = 0$, there are no nontrivial curves of genus 0, and the only smooth projective curve over $\Spec \ZZ$ is $\PP^1$ \cite{Fontaine85}.

\begin{remark}
The classification of smooth projective surfaces over $\ZZ$ is an open problem; it is not known whether there is any such surface of non-negative Kodaira dimension.  For further discussion of this intriguing question, and some partial results, see \cite{Khare07, Schroer23, BMP26}. Other sources of smooth projective varieties and Deligne--Mumford stacks over $\ZZ$ include smooth toroidal compactifications of $\cA_g$ and certain other Shimura varieties.
\end{remark}

The topological prediction that $H^k(\Mbar_{g,n}) = 0$ for odd $k < 11$ is proved in \cite{BFP24}. At present, all of the known proofs use algebraic geometry in an essential way. In particular, the proofs follow the inductive method of Arbarello and Cornalba \cite{ArbarelloCornalba98}, which relies on mixed Hodge theory and the yoga of weights. Moreover, the base cases needed to run the induction also seem out of reach by topological methods; they are proved using point counts over finite fields, the Grothendieck--Lefschetz trace formula, and \'etale-to-singular comparison theorems, as in \cite{BFP24}, or using the Chow--K\"unneth generation property, as in \cite{CanningLarson22, CLP24}.

The prediction that $H^k(\Mbar_{g,n})$ is pure Tate for even $k \leq 20$ is confirmed for $k \leq 14$ \cite{CLP24}, but the cases 16, 18, and 20 remain open.

\section{Further results and predictions} \label{sec:further-results}

The predictions of Chenevier--Lannes, and the methods used to prove these predictions for $H^k(\Mbar_{g,n})$, for $k \leq 15$, also shed light on the following stronger forms of the Hodge conjecture and Tate conjecture.

\begin{proof}[{\bf Sketch of proof of Theorems~\ref{thm:b} and \ref{thm:c}}]
In even degrees, if $H^{2k}(\Mbar_{g,n})$ is tautological, then it is generated by the classes of finitely many tautological cycles of codimension $k$, and this confirms the generalized Hodge and Tate conjectures. In particular, the even degree cases of Theorem~\ref{thm:b} follow immediately from the proof of Theorem~\ref{thm:a}.

We now consider the odd degree cases. For odd $k < 11$, the conjectures are vacuous, since $H^k(\Mbar_{g,n})$ vanishes, as predicted by Chenevier and Lannes (Theorem~\ref{thm:CL}) and proved unconditionally in \cite{BFP24}.

For $k = 11$, Theorem~\ref{thm:CL} predicts that $H^{11}(\Mbar_{g,n})$ should be a direct sum of copies of $\s_{12}$. In particular, we expect $H^{11}(\Mbar_{g,n}, \CC) \cong H^{11,0}\oplus H^{0,11}$, and no Galois subrepresentation of $H^{11}((\Mbar_{g,n})_{L^s}, \QQ_\ell)$ has positive Tate coniveau, unless $L$ has characteristic $p > 0$ and $p | \tau(p)$.  The decomposition is proved unconditionally in \cite{CanningLarsonPayne23}, where it is shown that
$H^{11}(\Mbar_{g,n})$ vanishes unless $g = 1$ and $n \geq 11$, and
\begin{equation}\label{eq:11}
H^{11}(\Mbar_{1,n}) \cong \bigoplus_{1 \in A, |A| = 11} \pi_A^* H^{11}(\Mbar_{1,11})
\end{equation}
for $n > 11$. Here, $\pi_A$ denotes the tautological morphism obtained by forgetting all marked points that are not in a specified subset $A \subset \{1, \ldots, n\}$. The direct sum decomposition \eqref{eq:11}, together with Remark~\ref{rmk:effectivity} and the semisimplicity of $H^{11}((\Mbar_{1,11})_{L^s}, \QQ_\ell)$, proves this case of the semisimplicity and generalized Tate conjectures.

The proof for $k = 13$ is similar, in the sense that we have a complete understanding of $H^{13}(\Mbar_{g,n})$ for all $g$ and $n$, as codimension-one Gysin push-forwards of copies of $\s_{12}$  \cite[Theorem~1.7]{CLPW24}.  For $g \geq 1$, and for $A \subset \{1, \ldots, n\}$ such that $2g-3+|A^c| > 0$, let $\xi_A$ be the tautological morphism
\[
\xi_A \colon \Mbar_{1, A \cup \{p \}} \times \Mbar_{g-1, A^c \cup \{q \}} \to \Mbar_{g,n}
\]
obtained by gluing the marked points $p$ and $q$. We restrict the Gysin push-forward to the K\"unneth summand from the first factor:
\[
\begin{aligned}
\xi_{A*} \colon &\big(H^{11}(\Mbar_{1, A \cup \{p \}}) \otimes H^0(\Mbar_{g-1, A^c \cup \{q \}})\big)(-1) \\
&\longrightarrow H^{13}(\Mbar_{g,n}).
\end{aligned}
\]
By \cite[Theorem~1.7]{CLPW24}, the sum of Gysin push-forwards $\bigoplus_A \xi_{A*}$ is surjective, and is an isomorphism unless $g = 1$.  (For $g = 1$, the kernel is also completely understood; see \cite[Section~4.4]{CLPW24}.) In any case, from the surjectivity of $\bigoplus_{A} \xi_{A*}$, we immediately see that $H^{13}(\Mbar_{g,n})$ is supported on a codimension-one closed subset, namely the union of the boundary divisors that are the images of the gluing maps $\xi_A$.  We also see that the Hodge type is $$H^{13}(\Mbar_{g,n}, \CC) = H^{12,1} \oplus H^{1,12}.$$ Similarly, the Galois representation $H^{13}((\Mbar_{g,n})_{L^s}, \QQ_\ell)$ is a quotient of a direct sum of Tate twists of $\s_{12}$. With the additional assumption that either $\ch(L) = 0$ or $\ch(L) = p$ and $p \nmid \tau(p)$,  it follows that every nontrivial Galois subrepresentation $V$ has Tate coniveau exactly 1. Since each such $V$ is supported on the boundary divisors in the image of the maps $\xi_A$ as above, the generalized Tate conjecture for $H^{13}((\Mbar_{g,n})_{L^s}, \QQ_\ell)$ is confirmed.  Moreover, the semisimplicity of $H^{13}((\Mbar_{g,n})_{L^s}, \QQ_\ell)$ follows from the surjectivity of $\bigoplus_A \xi_{A*}$ and the semisimplicity of $H^{11}((\Mbar_{1,11})_{L^s}, \QQ_\ell)$.

It remains to discuss $H^{2d-j}$ for odd $j \leq 15$.  By  Poincar\'e duality and \cite{BFP24}, these groups vanish for odd $j < 11$. For $j = 11$ and $13$, the relevant classes lie in the semi-tautological extension generated by $H^{11}(\Mbar_{1,11})$ \cite[Theorem~1.6]{CLP24}, which means that these classes are obtained from boundary strata supporting $\s_{12}$, via Gysin push-forward and cup product with Chern classes of tautological bundles. By representing the Chern classes of the restrictions of these bundles by algebraic cycles on the boundary strata, we see that the classes are supported on closed algebraic subsets of the predicted codimension. The case $j=15$ also uses $H^{15}(\Mbar_{1,15})$ \cite[Theorem~1.7]{CLP24}. Since $W_{15}H^{15}(\cM_{1,15}) = \s_{16}$, the $\mathsf{L}^2\s_{12}$ part of $H^{15}(\Mbar_{1,15})$ lies in the boundary Gysin image \cite[Lemma~8.3]{CLP24}. By the K\"unneth formula and the vanishing of odd cohomology below degree 11, the degree 13 cohomology of the normalized boundary has contributions of the form $H^{13}\otimes H^0$ and $H^{11}\otimes H^2$, up to interchanging factors. These are supported on divisors, by the degree 13 result above and the fact that $H^2$ is generated by divisor classes. Their Gysin images are therefore supported in codimension two. The same tautological operations then give the predicted support for $j=15$.
\end{proof}

Up to this point, we have discussed predictions and results regarding the images of cycle class maps and the support of Hodge and Galois substructures. We now turn to a complementary question: the expected structure of the kernel of the cycle class map, which is conjecturally related to filtrations on Chow groups (see, e.g., \cite{Jannsen94, MurreNagelPeters13}).

Recall that for a smooth complex projective variety (or smooth and proper Deligne--Mumford stack $X$ with projective coarse moduli space), we write $Z^k(X)$ for the group of codimension $k$ algebraic cycles on $X$.  The codimension $k$ Chow group $A^k(X)$ is the quotient of $Z^k(X)$ by the subgroup of cycles rationally equivalent to zero. Let $A^k(X)_\QQ := A^k(X) \otimes \QQ$. Then Bloch and Beilinson conjecture the existence of a decreasing filtration
\[
A^k(X)_\QQ = F^0A^k(X)_\QQ \supset F^1A^k(X)_\QQ \supset \cdots \supset F^{k+1}A^k(X)_\QQ = 0
\]
that is stable under correspondences, with the following properties:
\begin{enumerate}
    \item The kernel of the cycle class map $\cyc^k \colon A^k(X)_\QQ \to H^{2k}(X, \QQ)$ is $F^1A^k(X)_\QQ$.
    \item For $1 \leq i \leq k$, if $H^{p,q}(X) = 0$ for all $p + q = 2k - i$ and $q \leq k - i$, then
    \[
    F^iA^k(X)_\QQ/F^{i+1}A^k(X)_\QQ = 0.
    \]
\end{enumerate}
See \cite[Conjecture~11.21]{Voisin03}.

For $\Mbar_{g,n}$, recent results show that $H^{p,q} = 0$ for $p + q \leq 10$ and $q < p$ \cite{BFP24, CLP24}. If $1 \leq k \leq 5$ and $1 \leq i \leq k$, then $p+q = 2k-i$ and $q \leq k-i$ imply $p+q \leq 9$ and $q < p$. Thus property~(2) would give $F^iA^k(\Mbar_{g,n})_\QQ = F^{i+1}A^k(\Mbar_{g,n})_\QQ$ for every $i = 1, \ldots, k$. Together with $F^{k+1} = 0$, this would imply $F^1A^k(\Mbar_{g,n})_\QQ = 0$ and hence injectivity of the cycle class map. Since $H^{2k}(\Mbar_{g,n})$ is of Tate type for $2k\le 10$, the Hodge conjecture says that the cycle class map should also be surjective in these cases.  This leads to the following prediction.

\begin{Chow5}
The rational cycle class map $A^k(\Mbar_{g,n})_\QQ \to H^{2k}(\Mbar_{g,n})$ is an isomorphism for $k \leq 5$, and for all $g$ and $n$.
\end{Chow5}

\noindent This is trivial for $k = 0$, and for $k = 1$ it follows from \cite{ArbarelloCornalba98}.  As further evidence, we note that both the rational Chow groups $A^k(\Mbar_{g,n})_\QQ$ and the singular cohomology groups $H^{2k}(\Mbar_{g,n})$ are generated by tautological classes for many small values of $g$ and $n$ \cite{CanningLarson22}. In addition, $A^2(\Mbar_{g,n})_\QQ$ is generated by tautological cycles when the coarse space $\overline{M}_{g,n}$ is rationally connected \cite[Theorem~1.4]{Lu25}.

There is evidence that $A^6(\Mbar_{1,11})_\QQ$ may be infinite-dimensional. In particular, Lilienfeldt shows that Heegner cycles generate an infinite-rank subgroup of homologically trivial cycles modulo algebraic equivalence on Kuga--Sato varieties at level at least 5 \cite[Theorem~1.4]{Lilienfeldt24}. Deducing the corresponding statement for $\Mbar_{1,11}$ would require an additional argument to pass to level one. Establishing the corresponding statement at level one would show that the conjectured bound of $k \leq 5$ is best possible.

\begin{remark}
Note that all of these conjectures are formulated for Chow groups and singular cohomology groups with rational coefficients.  The Chow groups and singular cohomology groups of $\Mbar_{g,n}$ with integral coefficients are substantially less well-understood, and outside the scope of this survey.
\end{remark}

\section{Future outlook} \label{sec:outlook}

We now discuss some of the natural next steps toward proving the conjectures stated in this article and deepening our understanding of the Hodge and Tate conjectures for moduli spaces of curves, roughly in increasing order of expected difficulty.

For even degrees, we have seen that the generalized Hodge and Tate conjectures hold for $H^k(\Mbar_{g,n})$ for even $k \leq 6$, except for $k = 6$ and $g \in \{7, 8, 9\}$.

\begin{problem}
Show that $H^6(\Mbar_{g,n})$ is tautological for all $g$ and $n$.
\end{problem}

\noindent Once we know that $H^6(\Mbar_{g,n})$ is tautological for all $g$ and $n$, it will follow by the Arbarello--Cornalba induction that $H^8(\Mbar_{g,n})$ is tautological for all but a small finite list of $g$, and for all $n$.

For odd degrees, we have seen that the generalized Hodge conjecture holds for $\Mbar_{g,n}$ in all odd degrees less than 15, and that the generalized Tate conjecture holds in those cases under the additional ordinarity hypotheses in Theorem~\ref{thm:c}. By \cite[Theorems 1.1 and 1.7]{CLP24}, we also know that each irreducible motivic structure that is a subquotient of $H^{15}(\Mbar_{g,n})$ is isomorphic to either $\s_{16}$ or $\mathsf{L}^2 \s_{12}$, as predicted by Theorem~\ref{thm:CL}. Moreover, the multiplicity of $\s_{16}$ is zero unless $g =1$ and $n \geq 15$ \cite[Theorem~1.3]{CLPW25}.

\begin{problem}
Prove the generalized Hodge conjecture for $H^{15}(\Mbar_{g,n})$.
\end{problem}

\noindent It would suffice to show that every occurrence of $\mathsf{L}^2\s_{12}$ comes from a codimension 2 push-forward of a degree 11 cohomology class on a product of smaller moduli spaces.

New phenomena arise in higher cohomological degrees. For instance, the irreducible $\ell$-adic Galois representation attached to $\s_{18}:= W_{17}H^{17}(\cM_{1,17})$ appears in the middle cohomology of both $\Mbar_{1,17}$ and $\Mbar_{2,14}$. This is the first instance of a general phenomenon, related to Saito--Kurokawa lifts of modular forms. The cohomology of $\Mbar_{1,n}$ is related to modular forms for $\SL_2(\ZZ)$ by the Eichler--Shimura correspondence, and the cohomology of $\Mbar_{2,n}$ is related to the cohomology of local systems on $\cA_2$, which are in turn related to Siegel modular forms.  Some Siegel modular forms arise as Saito--Kurokawa lifts of modular forms from $\SL_2(\ZZ)$ to $\Sp_4(\ZZ)$, and the corresponding $\ell$-adic Galois representations are isomorphic, without any clear geometric explanation.

Gu, Lee, and Prasanna have announced work in preparation showing that the rational Hodge structure on $W_{17} H^{17}(\cM_{2,14})$ agrees with that on $\s_{18}$ \cite{GLP26}. One expects similar agreement of Hodge structures in other instances where isomorphic Galois representations appear on different moduli spaces, but this is not known.

We know that every holomorphic 17-form on $\Mbar_{g,n}$ is pulled back, by forgetting marked points, from either $\Mbar_{1,17}$ or $\Mbar_{2,14}$ \cite{CLPW25b}. From Theorem~\ref{thm:CL}, we expect that all other irreducible Hodge structures appearing in degree 17 on $\Mbar_{g,n}$ should have positive Hodge coniveau. The following is therefore a natural first step toward proving the generalized Hodge conjecture for $H^{17}(\Mbar_{g,n})$.

\begin{problem}
Show that $H^{17}(\Mbar_{g,n})$ is supported on a divisor, for $g \geq 3$ and all $n$.
\end{problem}

Another natural next direction to explore is the generalized Tate conjecture for $H^{11}((\Mbar_{g,n})_{\overline \FF_p}, \QQ_\ell)$ for primes $p \neq \ell$ such that $p$ divides $\tau(p)$.

\begin{problem} \label{prob:non-ordinary}
Show that $H^{11}((\Mbar_{1,11})_{\overline {\FF}_p}, \QQ_\ell)$ is supported on a closed algebraic subset of codimension $\nu_p(\tau(p))$, for $p \in \{ 2,\, 3,\, 5,\, 7,\, 2411,\, 7758337633\}$.
\end{problem}

\noindent The primes listed are the non-ordinary primes less than $10^{10}$ for $\s_{12}$, and the multiplicity $\nu_p(\tau(p))$ with which $p$ divides $\tau(p)$ in these cases is 1, except for $\nu_2(\tau(2)) = 3$ and $\nu_3(\tau(3)) = 2$. This problem illustrates a feature absent from the Hodge-theory side: in positive characteristic, the predicted coniveau can jump at non-ordinary primes, and the support problem becomes sensitive to $p$-adic slopes of modular Galois representations. Non-ordinary primes for higher cusp-form motives give further support problems analogous to Problem~\ref{prob:non-ordinary}; for instance, the non-ordinary primes for $\s_{16}$ less than $10^6$ are $\{2,\, 3,\, 5,\, 7,\, 11,\, 13,\, 59,\, 15271,\, 187441\}.$

Other difficult phenomena arise when one considers products of different moduli spaces of curves, such as $\Mbar_{1,17} \times \Mbar_{2,14}$, where isomorphic motivic structures appear independently, without geometric explanation.

\begin{problem}
After choosing an isomorphism between the copies of $\s_{18}$ that appear in $H^{17}(\Mbar_{1,17})$ and $H^{17}(\Mbar_{2,14})$, respectively, produce an algebraic cycle of dimension 17 whose class spans the Tate subspace $$\mathsf{L}^{17} \cong \wedge^2 \s_{18}  \subset H^{34}(\Mbar_{1,17} \times \Mbar_{2,14}).$$
\end{problem}

\noindent The existence of such a cycle is predicted by the Hodge and Tate conjectures, but there is no known natural candidate.

One also expects a nontrivial Hodge class in middle degree on $\Mbar_{3,16}$, coming from a tensor product of two irreducible motivic structures isomorphic to $\s_{12}$ whose $\ell$-adic realizations are related by an Ikeda lift. Finding algebraic cycles that represent such classes would be profoundly interesting; proving non-existence of such cycles would disprove the Hodge conjecture.

\bibliographystyle{amsalpha}
\bibliography{HodgeTateSurvey}
\end{document}